\documentclass[12pt]{amsart}
\usepackage{amssymb, amsrefs}
\usepackage{amscd}
\usepackage{a4wide}

\theoremstyle{plain}

\newtheorem{thm}{Theorem}
\newtheorem{cor}[thm]{Corollary}
\newtheorem{lem}[thm]{Lemma}

\newtheorem{prop}[thm]{Proposition}

\newtheorem{rem}[thm]{Remark}
\newtheorem{example}[thm]{Example}
\newcommand{\Ann}{\mathrm{Ann}}

\newcommand{\RR}{\mathbb{R}}

\newcommand{\NN}{\mathbb{N}}

\newcommand{\Soc}[1]{\mathrm{Soc}({#1})}
\newcommand{\SocZ}[1]{\mathrm{Soc}_2({#1})}

\newcommand{\m}{\mathfrak{m}}

\newcommand{\Max}[1]{\mathrm{Max}(#1)}
\newcommand{\gr}[1]{\mathrm{gr}(#1)}

\title[Simple Modules Over quasi-local rings]{A note on simple modules over quasi-local rings.}
\dedicatory{Dedicated to John Clark}

\date{\today}
\author{Paula A.A.B. Carvalho}
\author{Christian Lomp}
\address{Department of Mathematics of the Faculty of Science of the University of Porto, Rua Campo
Alegre, 687, 4169-007 Porto, Portugal}

\author{Patrick F. Smith}
\address{Department of Mathematics, University of Glasgow, Glasgow, United Kingdom}

\begin{document}
\maketitle
\begin{abstract}
Matlis showed that the injective hull of a simple module over a commutative Noetherian ring is Artinian. Many non-commutative Noetherian rings whose injective hulls of simple modules are locally Artinian have been extensively studied in \cite{BrownCarvalhoMatczuk,CarvalhoLompPusat, CarvalhoMusson, CarvalhoHatipogluLomp, Hatipoglu,HatipogluLomp, Jategaonkar, Musson, SantanaVinciguerra}. This property had been denoted by property $(\diamond)$. In this paper we investigate, which non-Noetherian semiprimary commutative quasi-local rings $(R,\m)$ satisfy property $(\diamond)$. 
For quasi-local rings  $(R,\m)$ with $\m^3=0$, we prove a characterisation of this property in terms of  the dual space of $\Soc{R}$.  Furthermore, we show that $(R,\m)$ satisfies $(\diamond)$ if and only if its associated graded ring $\gr{R}$ does. 

Given a field $F$ and vector spaces $V$ and $W$ and a symmetric bilinear map $\beta:V\times V\rightarrow W$ we consider commutative quasi-local rings of the form $F\times V \times W$, whose product is given by 
$(\lambda_1, v_1,w_1)(\lambda_2,v_2,w_2) = (\lambda_1\lambda_2, \lambda_1v_2+\lambda_2v_1, \lambda_1w_2+\lambda_2w_1+\beta(v_1,v_2))$ in order to build new examples and to illustrate our theory. In particular we prove that any quasi-local commutative ring with radical cube-zero does not satisfy $(\diamond)$ if and only if it has a factor, whose associated graded ring is of the form  $F\times V \times F$ with $V$  infinite dimensional and $\beta$  non-degenerated.
\end{abstract}

\section{Introduction}
The structure and in particular finiteness conditions of injective hulls of simple modules have been widely studied. Rosenberg and Zelinsky's work \cite{RosenbergZelinsky} is one of the earliest studies of finiteness conditions on the injective hull of a simple module. Matlis showed in his seminal paper \cite{Matlis} that any injective hull of a simple module over a commutative Noetherian module is Artinian. Jans in \cite{Jans} has termed a ring $R$ to be left co-noetherian if every simple left $R$-module has an Artinian injective hull. Vamos showed in \cite{Vamos} that a commutative ring $R$ is co-noetherian if and only if $R_\m$ is Noetherian for any maximal ideal $\m \in \Max{R}$ - generalizing in this way Matlis' result.
In connection with the Jacobson Conjecture for non-commutative Noetherian rings, 
Jategaonkar showed in \cite{Jategaonkar} (see also \cite{Cauchon, Schelter}) that the injective hulls of simple modules are locally Artinian, i.e. any finitely generated submodule is Artinian, provided the ring $R$ is fully bounded Noetherian.
We say that a ring $R$ satisfies condition $(\diamond)$ if 
\begin{equation}\tag{$\diamond$}\label{diamond_condition}
\mbox{Injective hulls of simple left $R$-modules are locally Artinian.}
\end{equation}
In this paper we study ($\diamond$) for, not necessarily Noetherian, quasi-local commutative rings $R$ with maximal ideal $\m$ such that $\m^3=0$. A description of such rings is given in terms of the dual space of $\Soc{R}$ seen as a vector space over $R/\m$ (Theorem \ref{main}). Furthermore, we relate property $(\diamond)$ of $(R,\m)$ with its associated graded ring $\gr{R}=R/\m \oplus \m/\m^2 \oplus \m^2$ in Corollary \ref{associatedgraded}. Given a field $F$ and vector spaces $V$ and $W$ and a symmetric bilinear map $\beta:V\times V\rightarrow W$ we consider commutative quasi-local rings of the form $F\times V \times W$, whose product is given by 
$$(\lambda_1, v_1,w_1)(\lambda_2,v_2,w_2) = (\lambda_1\lambda_2, \lambda_1v_2+\lambda_2v_1, \lambda_1w_2+\lambda_2w_1+\beta(v_1,v_2))$$
to build new examples and to illustrate our theory. In particular we prove in Proposition \ref{badfactor} that a quasi-local commutative ring with radical cube-zero does not satisfy $(\diamond)$ if and only if it has a factor whose associated graded ring is of the form $F\times V \times F$ with $V$ of infinite dimension and $\beta$ non-degenerate.

\section{Preliminaries}

The following Lemma shows that conditon $(\diamond)$ is intrinsically linked to Krull's intersection Theorem:

\begin{lem}\label{KrullIntersection}
Let $R$ be a (not necessarily commutative) ring with Jacobson radical $J$, such that finitely generated Artinian modules
have finite length. 
If $R$ has property $(\diamond)$, then for any left ideal $I$ of $R$ one has
$$\bigcap_{n=0}^\infty (I+J^n)  = I$$
\end{lem}

\begin{proof}
Let $I$ be any left ideal of $R$. Then R/I embeds into a product of cyclic modules $R/K_i$ with essential simple socle by
Birkhoff's theorem, where $I\subseteq K_i$ and the intersection $\bigcap K_i = I$.  By hypothesis each of these modules
$R/K_i$ is Artinian and hence has finite length.
Thus there exists a number $n_i\geq 1$ such that $J^{n_i}R/K_i =0 \Leftrightarrow  J^{n_i} \subseteq K_i.$
 Hence $I+J^{n_i} \subseteq K_i$ for all $i$ and as the intersection of the $K_i$'s is $I$, we have
$$I=\bigcap_i K_i \supseteq \bigcap_i \left(I+J^{n_i}\right) \supseteq \bigcap_n  (I+J^n) \supseteq I.$$
\end{proof}

\begin{rem}\label{remark}
\begin{enumerate}

\item  Assuming the hypotheses of Lemma \ref{KrullIntersection}, one can easily adapt the above proof to show that $\bigcap_{n=0}^\infty J^nM = 0$ for any finitely generated left $R$-module $M$. Furthermore, if $M$ is a finitely generated essential extension of a simple left $R$-module, then there exists $n>0$ such that $J^n M = 0$.

\item Finitely generated Artinian left $R$-modules have finite length if for example $R$ is left Noetherian or if $R$ is
commutative. For the later case let M be an Artinian module over a commutative ring R generated by $x_1, \ldots, x_k$.
Then 
$$R/Ann(M) \rightarrow R(x_1, ..., x_k) \subseteq M^k$$  is an embedding. Since $M^k$ is Artinian, $R/Ann(M)$ is
Artinian and by the Hopkins-Levitzki's Theorem $R/Ann(M)$ is Noetherian. As M is finitely generated over $R/Ann(M)$, M is
also Noetherian, i.e. has finite length.

\item We follow the terminology of commutative ring theory and call a commutative ring $R$ {\it quasi-local} if it has a unique maximal ideal $\m$. A local ring is a commutative Noetherian quasi-local ring. 
From Lemma \ref{KrullIntersection} we see that any  commutative quasi-local ring  $(R,\m)$, that satisfies $(\diamond)$, is separated in the
$\m$-adic topology. Moreover if $\m^n$ is idempotent for some $n\geq 1$, then $\m^n=0$. 
\end{enumerate}
\end{rem}

Recall that a ring $R$ with Jacobson radical $J$ is called {\it semilocal} if $R/J$ is semisimple. A semilocal ring with nilpotent Jacobson radical is called {\it semiprimary}.
The second socle of a module $M$ is the submodule $\SocZ{M}$ of $M$ with $\Soc{M/\Soc{M}} = \SocZ{M}/\Soc{M}$.
For an ideal $K$ of $R$, denote by $\Ann_M(K)$ the set of elements $m\in M$ such that $Km=0$. For a ring $R$ with $R/J$ semisimple, it is well-known that $\Ann_M(J)=\Soc{M}$ and that $\Ann_M(J^2)=\SocZ{M}$. For a left ideal $I$ set $(I:K)=\{r\in R\mid Kr\subseteq I\}$.

\begin{prop}\label{semiprimary} The following statements are equivalent, for a semiprimary ring $R$ with Jacobson radical $J$. 
\begin{enumerate}
\item[(a)] $R$ has property $(\diamond)$.
\item[(b)] $\SocZ{M}$ has finite length, for any left $R$-module $M$ with $\Soc{M}$ finitely generated.
\item[(c)] $(I:J^2)/I$ is finitely generated, for any left ideal $I$ of $R$ with $(I:J)/I$ finitely generated.
\end{enumerate}
\end{prop}

\begin{proof}
Note first that since $R$ is semiprimary, there exists $n\geq 0$ such that $J^n=0$. Moreover since $R$ is semilocal, $\Soc{M}=\mathrm{Ann}_M(J) = \{m \in M\mid Jm=0\}$ for any left $R$-module $M$.
In particular any left $R$-module has an essential socle, since any  left $R$-module $M$ has a finite socle series:
$$ 0=J^nM \subseteq J^{n-1}M \subseteq \cdots \subseteq JM \subseteq M.$$ 
$(a)\Rightarrow (b)$ If $R$ satisfies $(\diamond)$ then any finitely generated module with finitely generated (essential) socle is Artinian. Hence if $\Soc{M}$ is finitely generated, $M$ must be Artinian and hence $M/\Soc{M}$ is Artinian.

$(b)\Rightarrow (c)$ for $M=R/I$ one has $\Soc{R/I} = (I:J)/I$ as mentioned above. Moreover, $(I:J^2)=((I:J):J)$ and therefore $(I:J^2)/(I:J) = \Soc{R/(I:J)} = \Soc{M/\Soc{M}}$. Thus the statement follows from $(b)$.

$(c)\Rightarrow (a)$ is clear since if $I$ is a left ideal such that $M=R/I$ is a cyclic essential extension of a simple left $R$-module, then $(I:J)/I$ is cyclic and by assumption $(I:J^2)/I$ is finitely generated. Hence $(I:J^2)/I$ has finite length. Applying our hypothesis to $I'=(I:J)$, we can conclude that $(I:J^3)/I$ has finite length. Continuing we have also that $R/I=(I:J^n)/I$ has finite length. 
\end{proof}

A sufficient condition for a ring to satisfy $(\diamond)$ is given by the following Lemma.
\begin{lem}\label{soc-artinian} Any ring $R$, with $R/\Soc{_RR}$ Artinian, satisfies $(\diamond)$.
\end{lem}

\begin{proof}
Suppose $I\subset K \subseteq R$ are ideals such that $K/I$ is a simple $R$-module and essential in $R/I$. If $ \Soc{R}\subseteq I$, then $R/I$ is a factor of $R/\Soc{R}$  and hence Artinian. If $ \Soc{R}\nsubseteq  I$, then $( \Soc{R}+I)/I$ is a semisimple submodule of $R/I$ and hence must equal $K/I$, i.e. $ \Soc{R}+I=K$. As  a quotient of $R/\Soc{R}$, the module  $R/K=R/(\Soc{R}+I)$ is a Artinian and so is $R/I$.\end{proof}

Clearly it is not necessary for a ring $R$ with $(\diamond)$ to satisfy $R/\Soc{_RR}$ being Artinian. Moreover, Example \ref{ex_integral} shows that there are commutative rings $R$ such that $R/\Soc{R}$ satisfies $(\diamond)$, but $R$ does not. 

In recent papers \cite{CarvalhoLompPusat, CarvalhoHatipogluLomp, CarvalhoMusson, HatipogluLomp}, several non-commutative Noetherian rings have been shown to satisfy $(\diamond)$. In this note we intend to study condition  $(\diamond)$ for non-Noetherian commutative rings.

\section{Local-Global Argument}

Jans in \cite{Jans}{ defined a ring $R$ to be left co-noetherian if for every simple left $R$-module its injective hull is Artinian. Vamos has shown in \cite{Vamos} that 
a commutative ring $R$ is co-noetherian if and only if $R_\m$ is Noetherian for all $\m\in\Max{R}$.

The following Lemma shows the relation between co-Noetherianess and condition $(\diamond)$ for commutative quasi-local rings.  The proof follows the ideas of \cite{Smith}*{Theorem 1.8}.

\begin{lem}
The following statements are equivalent for a commutative quasi-local ring $R$ with maximal ideal $\m$.
\begin{enumerate}
\item[(a)] $R$ is Noetherian.
\item[(b)] $R$ is co-noetherian.
\item[(c)] $R$ satisfies $(\diamond)$ and $\m/\m^2$ is finitely generated.
\item[(d)] $R$ satisfies $\bigcap_{n=0}^\infty (I+\m^n)  = I$ for all ideals $I$ of $R$ and $\m/\m^2$ is finitely
generated.
\end{enumerate}
\end{lem}

\begin{proof}
$(a)\Leftrightarrow (b)$ follows from Vamos' result \cite{Vamos}*{Theorem 2}.

$(a)\Rightarrow (c)$ is clear and $(c)\Rightarrow (d)$ follows from Lemma \ref{KrullIntersection}.

$(d)\Rightarrow (a)$  There exists a finitely generated ideal $B$ of $R$ such that $\m = B + \m^2$. Then
$$\m= B + \m^2  = B+ (B+\m^2)^2 \subseteq B + \m^3 \subseteq \m,$$
and in general $\m = B + \m^n$ for every positive integer $n$. It follows that $\m = B$ and hence $\m$ is finitely
generated.
Suppose that $R$ is not Noetherian. Let $Q$ be maximal among the ideals $C$ of $R$ such that $C$ is not finitely
generated. Then $Q$ is a prime ideal of $R$ by a standard argument (see \cite{Cohen}*{Theorem 2}). Clearly $Q \neq \m$.  
Let $p \in \m$ with $p \notin Q$. By the choice of $Q$ the ideal $Q + Rp$ is finitely generated, say
\begin{equation}\label{eq_Q}
Q + Rp = R(q_1 + r_1p) + \dots + R(q_k + r_kp),
\end{equation}
for some positive integer $k$, $q_i \in Q (1 \leq i \leq k)$, $r_i \in R (1 \leq i \leq k)$. Let $D = Rq_1 + \cdots + Rq_k \subseteq Q$. Let $q\in Q\setminus D$. Then by equation (\ref{eq_Q}) there exist $s_1,\ldots, s_k\in R$ and $d\in D$ such that
$ q - d  =  (s_1r_1+\cdots +s_kr_k)p \in Q$. Since $Q$ is prime and $p\not\in Q$, $(s_1r_1+\cdots +s_kr_k)\in Q$, i.e. $Q = D + Qp$. Now $Q = D + Qp^t$ for every positive integer $t$ and
hence
$$Q = \bigcap_{s = 1}^{\infty} (D + Qp^s) \subseteq \bigcap_{s = 1}^{\infty} (D + \m^s) = D.$$
It follows that $Q = D$ and hence $Q$ is finitely generated, a contradiction. Thus $R$ is Noetherian.

\end{proof}

The diamond condition is equivalent to the condition that any injective hull of a simple  $R$-module  is locally Artinian, i.e. finitely generated submodules are Artinian. As seen in Remark \ref{remark}(2) a commutative ring satisfies $(\diamond)$ if and only if any injective hull $E$ of a simple $R$-module is locally of finite length, i.e. any finitely generated submodule of $E$ has finite length.
 
The following Proposition is well-known and can be found for example in \cite{SharpeVamos}*{Proposition 5.6}

\begin{prop}\label{localizing}
Let $R$ be  commutative ring, $\m$ a maximal ideal of $R$ and denote by $R_\m$ the localisation of $R$ by $\m$.
Then the injective hull $E=E(R/\m)$ of $R/\m$  as $R$-module is also the injective hull of $R_\m/\m R_\m$ as $R_\m$-module.
\end{prop}

It would be good to have a kind of local-global argument for condition $(\diamond)$ in comparison to Vamos' result on co-Noetherian rings. The following Proposition intends to find this kind of argument.

\begin{prop} Let $R$ be a commutative ring, $\m \in \Max{R}$ and $E=E(R/\m)$ the injective hull of $R/\m$. The
following statements are equivalent:
\begin{enumerate}
 \item[(a)] $E$ is locally of finite length as $R_\m$-module;
 \item[(b)] $E$ is locally of finite length as $R$-module and the $R$-submodule generated by an element $x\in E$ and the $R_\m$-submodule generated by $x$ coincide.
\end{enumerate}
\end{prop}

\begin{proof}
$(a)\Rightarrow (b)$
Suppose $E$ is locally of finite length as $R_\m$-module. Let $0\neq x\in E$. By hypothesis
$R_m x$ has finite length as $R_\m$-module. Hence there exist $k>0$ such that $\m^k x \subseteq (\m R_\m)^k
x = 0$. For any  $a\in R\setminus \m$ we have $R = Ra + \m$, which implies also $R = Ra + \m^k$. Thus
there exists $b\in R$ such that $1-ab\in \m^k$, i.e. $x=abx$. Hence $a^{-1}x = bx \in Rx$ shows that $R_m x = Rx$.
Hence the $R$-submodule generated by any set of elements of $E$ coincides with the $R_\m$-submodule generated by that
set. In particular any finitely generated $R$-submodule of $E$ is also a finitely generated $R_\m$-submodule of $E$,
which has finite length.

$(b)\Rightarrow (a)$ The condition that $R$-submodules and $R_\m$-submodules generated by a set coincides means that the lattice of submodules of $_RE$ and $_{R_\m}E$ are identical. Hence $E$ is locally of finite length as $R$-module implies also that it is locally of finite length as $R_\m$-module.
\end{proof}

\begin{cor}
The following statements are equivalent for a commutative ring $R$.
\begin{enumerate}
\item[(a)]  $R_\m$ satisfies $(\diamond)$ for all $\m\in\Max{R}$.
\item[(b)] $R$ satisfies $(\diamond)$ and the $R$-submodule generated by an element $x\in E(R/\m)$ and the $R_\m$-submodule generated by $x$ coincide.
\end{enumerate}
\end{cor}

{\bf Question:} Does there exist a commmutative ring $R$ that satisfies $(\diamond)$, but $R_\m$ does not satisfy $(\diamond)$, for some $\m\in \Max{R}$ ?

\section{Commutative semiprimary quasi-local rings}

A quasi-local commutative ring is co-noetherian if and only if  it is Noetherian.
Recall, that an ideal $I$ of a ring $R$ is called {\emph{subdirectly irreducible}} if $R/I$ has an essential simple socle.
Clearly a ring $R$ satisfies $(\diamond)$ if and only if $R/I$ is Artinian, for all subdirectly irreducible ideals $I$ of $R$.

\subsection{Commutative quasi-local rings with square-zero maximal ideal}\label{trivialextension}

Given any vector space $V$ over a field $F$, the trivial extension (or idealization) is defined on the vector space $R=F\times V$ with multiplication given by $(a, v)(b, w)= (ab, aw+vb)$, for all $a,b\in F$ and $v,w\in V$. Any such trivial extension $R$ is a commutative quasi-local ring that satisfies $(\diamond)$. However $R$ is Noetherian if and only if $V$ is finite dimensional.

\begin{lem}\label{squarezero} Any commutative quasi-local ring with square-zero maximal ideal satisfies  $(\diamond)$.
\end{lem} 

\begin{proof}
Let $(R,\m)$ be a commutative quasi-local ring with $\m^2=0$. Since $\m$ is a vector space over $R/\m$ it is semisimple.
Let $K$ be any subdirectly irreducible ideal of $R$. If $K=\m$, then $R/\m$ is simple. So assume $K\subset \m$. Then there exists a complement L such that $\m=L\oplus K$ and $\Soc{R/K}=\m/K \simeq L$ is simple.
$$\begin{CD} 0 @>>> \m/K @>>> R/K @>>>  R/\m @>>>  0 \end{CD}$$
is a short exact sequence. Hence $R/K$ has length $2$.
\end{proof}

\subsection{Commutative quasi-local rings with cube-zero maximal ideal}
In this section we will characterise commutative quasi-local rings $(R,\m)$ with $\m^3=0$ satisfying $(\diamond)$. Recall that $\Soc{R}=\mathrm{Ann}(\m)=\{r\in R: r\m=0\}$. Hence $\m^2\subseteq \Soc{R}$. We start with a simple observation.
\begin{lem}\label{cube-zerolocal}
Let $(R,\m)$ be a commutative quasi-local ring with $\m^3=0$.
\begin{enumerate}
\item[(1)] If $\m/\Soc{R}$ is finitely generated, then $R$  satisfies $(\diamond)$.
\item[(2)] If $\Soc{R}$ is finitely generated, then $R$ satisfies $(\diamond)$ if and only if $\m/\Soc{R}$ is finitely generated.
\end{enumerate}
\end{lem}

\begin{proof}
(1) If $\m/\Soc{R}$ is finitely generated, then $R/\Soc{R}$ is Artinian and by Lemma \ref{soc-artinian}, $R$ satisfies $(\diamond)$.

(2) If $\Soc{R}$ is finitely generated, then by Proposition \ref{semiprimary} $R$ satisfies $(\diamond)$ if and only if $\SocZ{R}=\m$ is Artinian if and only if $\m/\Soc{R}$ is finitely generated.
\end{proof}

The last Lemma raises the question, whether  the converse of (1) holds? That is, whether  $\m/\Soc{R}$ needs to be finitely generated for a commutative quasi-local ring $R$ with $\m^3=0$ and satisfying property $(\diamond)$? As we will see in Example \ref{example_Mio}, this need  not be the case.

\begin{lem}\label{correspondence}
Let $(R,\m)$ be a commutative quasi-local ring with residue field $F=R/\m$. Suppose $\m^3=0$. Then there exists a correspondence between subdirectly irreducible ideals of $R$ that do not contain $\Soc{R}$ and non-zero linear maps $f: \Soc{R}\rightarrow F$.
Each corresponding pair $(I,f)$ satisfies $$\mathrm{Soc(R)}+I=V_f:=\{a\in \m\mid f(\m a)=0\}.$$
\end{lem}

\begin{proof}
	Let $I$ be a subdirectly irreducible ideal that  does not contain $\Soc{R}$, then 
	$Soc(R/I) = (Soc(R)+I)/I$ is simple. Thus $\Soc{R} = Fx \oplus (\Soc{R}\cap I)$, for a non-zero element $x\in \Soc{R}$. 
	Let $f:\Soc{R}\rightarrow F$ be the linear map such that $f_{\mid_{\Soc{R}\cap I}}=0$ and $f(x)=1$.
	Clearly $\Soc{R}+I \subseteq V_f$, because $\m(\Soc{R}+I) = \m I\subseteq \m^2\cap I \subseteq  \Soc{R}\cap I$.
	
	
	To show that $V_f=\Soc{R}+I$ we use the essentiality of $\Soc{R/I} = (\Soc{R}+I)/I$ in $R/I$:
	For any $a\in V_f\setminus I$, there exists $r\in R$ such that $ra + I$ is a non-zero element of $Soc(R/I)=(\Soc{R}+I)/I$. Note that $r\not\in \m$ since otherwise $f(ra)=0$ and hence $ra\in \mathrm{Ker}(f)\subseteq I$. Therefore  $r$ is invertible and $a+I = r^{-1}ra+I  \in \Soc{R/I}$, i.e. $V_f=\Soc{R}+I$. 
	
	On the contrary, let $f$ be any non-zero element  $f\in\mathrm{Hom}_F(\Soc{R},F)$. Then there exists a non-zero element $x\in \Soc{R}$ with $f(x)=1$ and hence $\Soc{R}=Fx\oplus \mathrm{Ker}(f)$. Let $I$ be an ideal of $R$ that contains $\mathrm{Ker}(f)$ and that is maximal with respect to $x\not\in I$. Thus $I$ is subdirectly irreducible and $\Soc{R/I}=\Soc{R}+I=(Fx\oplus I)/I$ is simple and essential in $R/I$. By construction $\mathrm{Ker}(f) = I\cap \Soc{R}$. 
	
Note that $\m(\Soc{R}+I) = \m I \subseteq \m^2 \cap I \subseteq \Soc{R}\cap I = \mathrm{Ker}(f)$, i.e. $\Soc{R}+I\subseteq V_f$. To show the converse, let $a\in V_f\setminus I$, then by essentiality there exists $r\in R$ with $ra = x + y \in \Soc{R}+I$ with $x\in \Soc{R}$ and $y\in I$ and $ra\not\in I$. If $r\in \m$, then $ra\in \mathrm{Ker}(f)\subseteq I$, contradicting essentiality. Hence $r\not\in \m$ and $a\in \Soc{R}+I$, i.e. $V_f=\Soc{R}+I$.
\end{proof}

 \begin{thm}\label{main} Let $(R,\m)$ be a commutative quasi-local ring with residue field $F$ and $\m^3=0$. Then $R$ satisfies $(\diamond)$ if and only if $\m/V_f$ is finite dimensional for any $f\in\mathrm{Hom}_F(\Soc{R}, F)$.
\end{thm}

\begin{proof}
Suppose that $R$ satisfies $(\diamond)$ and let $f\in \mathrm{Hom}(\Soc{R},F)$. If $f=0$, then $V_f=\m$ and $\m/V_f$ has dimension zero. If $f\neq 0$, then by Lemma \ref{correspondence}, there exists a subdirectly irreducible ideal $I$ with $V_f=Soc(R)+I$ and $I$ not containing $\Soc{R}$. As $R$ satisfies $(\diamond)$, $R/I$ is Artinian and as a subquotient $\m/V_f = (\m/I)/(V_f/I)$ is also Artinian. As $R$-module $\m/V_f$ is semisimple since $\m^2\subseteq V_f$. Hence $\m/V_f$ must be finite dimensional.
	
Suppose $\m/V_f$ is finite dimensional for any $f\in\mathrm{Hom}(\Soc{R}, F)$.
Let $I$ be a subdirectly irreducible ideal of $R$.
If $\Soc{R}\subseteq I$, then $R/I$ is an $R/\m^2$-module. Since $R/\m^2$ is a quasi-local ring with square-zero radical, we have by Lemma \ref{squarezero}, that $R/\m^2$ satisfies $(\diamond)$. Hence $R/I$ must be Artinian.
If $\Soc{R}\not\subseteq I$, then by Lemma \ref{correspondence}, there exists a non-zero map $f:\Soc{R}\rightarrow F$ such that $\Soc{R}+I=V_f$. By hypothesis $\m/V_f$ is finite dimensional and is therefore Artinian as $R$-module. As $R/\m$ and $V_f/I$ are simple modules, also $R/I$ is Artinian, proving that $R/I$ is Artinian for any subdirectly irreducible ideal $I$ of $R$, i.e. $R$ satisfies $(\diamond)$.
\end{proof}

Let $(R,\m)$ be any commutative quasi-local ring. The {\it associated graded ring} of $R$ with respect to the $\m$-filtration is the commutative  ring
$\gr{R}=\bigoplus_{n\geq 0} \m^n/\m^{n+1}$ with  multiplication given by 
$$(a+\m^{i+1})(b+\m^{j+1}) =ab + \m^{i+j-1},\qquad \forall a\in \m^i, b\in \m^j, \:\: \mbox{and}\: i,j\geq 0.$$
For any ideal $I$ of $R$, the associated graded ideal is $\gr{I} = \bigoplus_{n\geq 0} (I\cap \m^n  + \m^{n+1})/\m^{n+1}.$
In particular $\gr{\m}=\bigoplus_{n\geq 1} \m^n/\m^{n+1}$ is the unique maximal ideal of $\gr{R}$. Hence $(\gr{R},\gr{\m})$ is a commutative quasi-local ring with residue field $F=R/\m$. Furthermore, $\gr{\Soc{R}}$ is contained in $\Soc{\gr{R}}$.

In case $\m^3=0$, the associated graded ring of $(R,\m)$ is $gr{R} = F \times \m/\m^2\times \m^2$, where $F=R/\m$ is the residue field of $R$. For any proper ideal $I$ of $R$ one has $$\gr{I}=0\times (I+\m^2)/\m^2 \times (I\cap \m^2)$$ and in particular $\gr{\Soc{R}}=0\times \Soc{R}/\m^2 \times \m^2=\Soc{\gr{R}},$ because for $(0,a+\m^2,b)\in \Soc{\gr{R}}$ we have that
$(0,a+\m^2,b)(0,x+\m^2,0) = (0,0,ax)=(0,0,0)$, for all $x\in \m$ if and only if $a\in \Soc{R}$. To shorten notation we will write the elements of $\gr{R}$ as 
$(a_0,a_1,a_2)$ for $a_i\in \m^i$, where the $i$th component is understood to be modulo $\m^{i+1}$.

\begin{lem}\label{compare_Vf} Let $(R,\m)$ be a commutative quasi-local ring with residue field $F$ and $\m^3=0$. For any  $f\in \mathrm{Hom}_F(\gr{\Soc{R}},F)$ there exists $g\in \mathrm{Hom}_F(\Soc{R},F)$ such that $$V_f = \{a\in \gr{\m}: f(\gr{\m}a)=0\} = \gr{V_g},$$ 
where $V_g$ is the ideal of $R$ defined by Lemma \ref{correspondence}.
\end{lem}

\begin{proof}
Let $f:\gr{\Soc{R}}\rightarrow F$ and denote by $\pi:\Soc{R}\rightarrow \m^2$ the projection onto $\m^2$, since $\m^2$ is a direct summand of $\Soc{R}$. Define $g:\Soc{R}\rightarrow F$ by $g(a)=f(0,a, \pi(a))$, for all $a\in \Soc{R}$. Then 
$(0,x,y)\in V_f$ if and only if $f(0,0,tx)=0$,  for all $t\in \m$. Since $tx=\pi(tx) \in \m^2$, the later is equivalent to $g(tx)=0$ for all $t\in \m$, i.e. $x\in V_g$ (in the ring $R$).
Hence $V_f = \{ (0,x,y) \in \gr{\m} \mid x\in V_g, \: y\in \m^2\} = \gr{V_g}$.
\end{proof}

\begin{cor}\label{associatedgraded} Let $(R,\m)$ be a quasi-local ring with $\m^3=0$. Then $R$ satisfies $(\diamond)$ if and only if its associated graded ring $\gr{R}$ does.
\end{cor}

\begin{proof}
If $R$ satisfies $(\diamond)$ and $f:\Soc{\gr{R}}\rightarrow F$ is a non-zero map, then by Lemma \ref{compare_Vf} there exists $g:\Soc{R}\rightarrow F$ such that $V_f=\gr{V_g}$. Since $V_g$ contains $\Soc{R}$ and hence $\m^2$, we have $V_f=\gr{V_g}=0\times V_g/\m^2 \times \m^2$. Thus, $\gr{\m}/V_f \simeq \m/V_g$. By Theorem \ref{main}, $\m/V_g$ is finite dimensional as $R$ satisfies $(\diamond)$. Hence $\gr{\m}/V_f$ is finite dimensional for all $f\in \Soc{\gr{R}}^*$. Again by Theorem \ref{main}, $\gr{R}$ satisfies $(\diamond)$.

Let $f:\Soc{R}\rightarrow F$ be any non-zero linear map and let $V_f=\{a\in \m  \mid f(\m a)=0\}$. 
If $\m^2 \subseteq \mathrm{Ker}(f)$, then $V_f=\m$.
If $\m^2\not\subseteq \mathrm{Ker}(f)$, then there exists $x\in \m^2$ with  $f(x)=1$.
Note that $V_f^2\subseteq \mathrm{Ker}(f)$, hence  $I = 0 \times V_f/\m^2 \times \mathrm{ker}(f)$ is a subdirectly irreducible ideal of $\gr{R}$. To see this note that  $E=0\times 0 \times Fx$ is a simple submodule of 
 $\gr{R}/I \simeq F \times \m/V_f \times Fx$. We will show that $E$ is essential in $\gr{R}/I$. 
 Let $(0,\overline{a},\overline{b}) \in \gr{R}/I$. If $a\not\in V_f$, there exists $c\in \m$ such that $f(ac)\neq 0$ and
 $ac - f(ac) x \in \mathrm{Ker}(f)$. Hence $(0,\overline{a},\overline{b})(0,\overline{c},0)=(0,0,f(ac)x) \in E$ is non-zero. 
 If $a\in V_f$, i.e. $\overline{a}=0$, and $\overline{b}\neq 0$, then $(0,0,\overline{b})$ is a non-zero element of $E$. 
 Hence $E$ is an essential simple submodule of $\gr{R}/I$ and if $\gr{R}$ satisfies $(\diamond)$, the quotient $\gr{R}/I$ and therefore also $\m/V_f$ must be Artinian, thus finite dimensional. By Theorem \ref{main}, $R$ satisfies $(\diamond)$.
\end{proof}

\section{Examples}

Let $(R,\m)$ be a commutative quasi-local ring with $\m^3=0$. The associated graded ring $\gr{R}$ is of the form $\gr{R} = F \oplus V \oplus W$ where $F=R/\m$ and $V=\m/\m^2$ and $W=\m^2$ are spaces over $F$. Moreover, the multiplication of $R$ induces a symmetric bilinear map $\beta:V\times V\rightarrow W$. Hence $\gr{R}$ is uniquely determined by $(F,V,W,\beta)$ and its multiplication can be identified with the multiplication of a generalized matrix ring. 
Writing the elements of $S=F\times V \times W$ as $3$-tuples $(\lambda, v,w)$ we have that the multiplication is given by
$$ (\lambda_1, v_1,w_1) (\lambda_2, v_2,w_2) = (\lambda_1\lambda_2,\lambda_1v_2+\lambda_2v_1,\lambda_1w_2 + \lambda_2w_1 + \beta(v_1,v_2)).$$
The units are precisely the elements $(\lambda,v,w)$ with $\lambda\neq 0$ and the unique maximal ideal of $S$ is given by $\mathrm{Jac}(S)=0\times V \times W$.
Let $$V^\perp_\beta = \{a\in V\mid \beta(V,a)=0\},$$ then $\Soc{S}=0\times V^\perp_\beta \times W$, while $\mathrm{Jac}(S)^2=0\times 0 \times \mathrm{Im}(\beta)$.

Recall, that $\beta$ is called non-degenerate or non-singular if $V^\perp_\beta=0$. In general $\beta$ need not be non-degenerate:

\begin{example}
	Let $F$ be any field and $V$ be any vector space over $F$ with countably infinite basis $\{v_0,v_1, v_2, \ldots\}$. Define a symmetric bilinear form as $\beta:V\times V \rightarrow F$ with 
	$\beta(v_0,v_0)=1$ and $\beta(v_i,v_j)=0$ for any $(i,j)\neq (0,0)$. Then $S=F\times V \times F$ is a commutative quasi-local ring that satisfies $(\diamond)$, because $V^\perp_\beta=\mathrm{span}(v_i\mid i>0)$ and hence 
	$$\mathrm{Jac}(S)/\Soc{S} = (0\times V \times W)/(0\times V^\perp_\beta \times F) \simeq F$$ is one-dimensional. By Lemma \ref{cube-zerolocal}, $S$ satisfies $(\diamond)$. Note that $S$ is not Artinian. Moreover, $S = \gr{R}$, where $R=F[x_0, x_1, x_2 \ldots]/ \langle x_0^3, \: x_ix_j \mid (i,j)\neq (0,0)\rangle.$
\end{example}

The bilinear form of the last example was not non-degenerate. Since $0\times V^\perp_\beta \times 0$ is always an ideal, we can pass to $F\times V/V^\perp_\beta \times W$ where the bilinear form $\beta$ is now non-degenerate. The following is a natural example of such a ring with non-degenerate bilinear form:

\begin{example}\label{ex_integral}  Let $F=\mathbb{R}$ and let $V=C([0,1])$ be the space of continuous real valued functions on $[0,1]$. Set
	$$\langle f,g \rangle  = \int_0^1 f(x)g(x) dx, \qquad \forall f,g \in C([0,1]).$$
	Then $\langle , \rangle:V\times V \rightarrow \mathbb{R}$ is a   non-degenerate symmetric bilinear form on $V$. Hence, by  Lemma \ref{cube-zerolocal}, $R=\mathbb{R}\times V \times \mathbb{R}$ is a commutative quasi-local ring with cube-zero radical, that does not satisfy $(\diamond)$, because its socle $\Soc{R}=0\times 0\times \RR$ is one-dimensional, hence finitely generated, but $\mathrm{Jac}(R)/\Soc{R}$ is infinite dimensional, hence not finitely generated as $R$-module.
\end{example}

These kind of rings must occur as the associated graded ring of a quotient of a commutative quasi-local ring $(R,\m)$ with $\m^3=0$ that does not satisfy $(\diamond)$.

\begin{prop}\label{badfactor} A commutative quasi-local ring $(R,\m)$ with $\m^3=0$ and residue field $F$ does not satisfy $(\diamond)$ if and only if it has a factor $R/I$ whose  associated graded ring $\gr{R/I}$ is of the form $F\times V \times F$ with $\mathrm{dim}(V)=\infty$ and non-degenerate bilinear form $\beta:V\times V\rightarrow F$.
\end{prop}

\begin{proof} By Theorem \ref{main} $R$ does not satisfy $(\diamond)$ if and only if there exists $f\in \Soc{R}^*$ such that $\m/V_f$ has infinite dimension. Note that $f\neq 0$ since $V_f\neq \m$. 
By Lemma \ref{correspondence} there exists a subdirectly irreducible ideal $I$ of $R$ such that $V_f=\Soc{R}+I$. In particular $(\m/I)^2 = V_f/I=\Soc{R/I}\simeq F$, because if $\m^2\subseteq I$, then $\m/I\subseteq \Soc{R/I}=V_f/I$ and hence $\m=V_f$, contradicting $V_f\neq \m$. Hence $m^2\not\subseteq I$ and $(\m/I)^2 = \Soc{R/I}$ as $R/I$ has a simple socle.
Moreover, $\gr{R/I}=F\times \m/V_f \times F$, with  bilinear form $\beta: \m/V_f\times \m/V_f \rightarrow F$, which is non-degenerate by the definition of $V_f$.

On the other hand if $R$ has a factor $R/I$ whose associated graded ring $\gr{R/I}$ is of the form $F\times V \times F$ with infinite dimensional vector space $V$ and non-degenerate bilinear form $\beta:V\times V\rightarrow F$, then $\Soc{\gr{R/I}}=0\times   0\times F$ is a simple submodule of $\gr{R/I}$, which is essential since $\beta$ is non-degenerate. As $V$ is infinite dimensional, the semisimple $\gr{R/I}$-module $(0\times V \times F)/(0\times 0 \times F)$ is not artinian and hence $\gr{R/I}$ does not satisfy $(\diamond)$. By Corollary \ref{associatedgraded},  $R/I$ does not satisfy $(\diamond)$ and therefore also $R$ does not.
\end{proof}

Let $V=A$ be any unital commutative $F$-algebra. Consider the multiplication of $A$ as a symmetric non-degenerate bilinear map $\mu:A\times A \rightarrow A$ and form the ring $S=F\times A \times A$ as before. In order to apply Theorem \ref{main}
recall that $\m=0\times A \times A$ and $\Soc{S} = 0\times 0\times A$, as the multiplication of $A$ is non-degenerate. Hence elements of $\Soc{S}^*$ can be identified with elements of $A^*$. For any $f\in A^*$ we defined 
$$V_f = \{(0,a,b)\in \m \mid f(Aa)=0 \} = 0 \times I(f) \times A,$$
where $I(f)$ is the largest ideal of $A$ that is contained in $\mathrm{Ker}(f)$. 
Theorem \ref{main} says that $S$ satisfies $(\diamond)$ if and only if $\m/V_f\simeq A/I(f)$ is finite dimensional for any $f\in A^*$.
From the theory of coalgebras, we borrow the notion of  the {\it finite dual} $A^\circ$ of an algebra, which is the subspace of $A^*$ consisting of the elements $f\in A^*$ that contain an ideal of finite codimension in their kernel. Hence $S$ satisfies $(\diamond)$ if and only if $A^\circ=A^*$.

\begin{example}\label{example_Mio}
The trivial extension $A=F\times V$ of a vector space $V$ (see \ref{trivialextension}) is an example of an algebra $A$ satisfying $A^\circ = A^*$. To see this, note that for any linear subspace $U$ of $A$, $U\cap V$ is an ideal of $A$. Thus, if $f\in A^*$, then $\mathrm{Ker}(f)\cap V$ is an ideal of codimension less or equal to $2$ and $f\in A^\circ$. In particular for such $A$, $S=F\times A\times A$ satisfies $(\diamond)$. However if $V$ is infinite dimensional, then $\m/\mathrm{Soc(S)}\simeq V$ is not finitely generated as $S$-module, which shows that the converse of Lemma \ref{cube-zerolocal}(1)  does not hold.
\end{example}

However, it might happen that the kernel of an element $f\in A^*$ does not contain an ideal of finite codimension as the following example shows.

\begin{example}
Let $A$ be a commutative unital $F$-algebra with multiplication $\mu$ and $f\in A^*$.
 Note that the composition $\beta=f\circ \mu$ is a non-degenerate bilinear form if and only if $f(Aa) \neq 0$, for all non-zero $a\in A$. Or, in other words, $\beta$ is non-degenerate if and only if  $\mathrm{Ker}(f)$ does not contain any non-zero ideal of $A$.
Such a map $f$ can be constructed in case $F$ has characteristic zero and $A$  has a countably infinite multiplicative basis $\{b_n\}_{n\in \NN}$, i.e. $b_n b_m = b_{n+m}$ for all $n,m \in \NN$. This implies in particular that $b_0=1$. Suppose $f:A\rightarrow F$ is a linear map and let $0\neq a\in A$ such that $f(Aa)=0$. For  $a=\sum_{i=0}^n \lambda_i b_i \in A$ and any $m\geq 0$, we have
$$f(b_ma) = \sum_{i=0}^n \lambda_i  f(b_{m+i}).$$
Thus $v=(\lambda_0, \lambda_1, \ldots, \lambda_{n})$ is in the kernel of the linear map given by the matrix:
$$B_n=\left(\begin{array}{ccccc} f(b_0) & f(b_1) & f(b_2) & \cdots & f(b_n) \\  
f(b_1) & f(b_2) & f(b_3) & \cdots & f(b_{n+1}) \\
 f(b_2) & f(b_3) & f(b_4) & \cdots & f(b_{n+2}) \\
\vdots & &  & \ddots & \vdots \\
   f(b_n) & f(b_{n+1}) & f(b_{n+2}) & \cdots & f(b_{2n+1})
 \end{array} \right)$$
In particular $\mathrm{det}(B_n)=0$. Hence if the sequence $(f(b_n))_{n\in \NN}$ produces a sequence of matrices $(B_n)$ that have all non-zero determinant, then for each  $0\neq a  =\sum_{i=0}^n \lambda_i b_i \in A$, there exists $0\leq m \leq n$ such that $f(b_ma) \neq 0$, i.e. $\beta = f\circ\beta$ is non-degenerate.

Matrices of the form of $B_n$ are called {\it Toeplitz} or {\it Hankel matrices}. A particular example of such a matrix is the {\it Hilbert matrix}, which is the matrix
$$B_{n-1}=\left(\begin{array}{ccccc} 1 & \frac{1}{2} & \frac{1}{3} & \cdots & \frac{1}{n} \\ \frac{1}{2} & \frac{1}{3} &\frac{1}{4} & \cdots & \frac{1}{n+1} \\
\frac{1}{3} & \frac{1}{4} & \frac{1}{5} & \cdots & \frac{1}{n+2} \\
\vdots &\ddots &  & \cdots & \vdots \\
\frac{1}{n} &\frac{1}{n+1} & \frac{1}{n+2} & \cdots & \frac{1}{2n-1}
 \end{array} \right)$$
in case $F$ has characteristic $0$. 
In 1894, Hilbert computed that $\mathrm{det}(B_{n-1}) = \frac{c_n^4}{c_{2n}} $, where $c_n = \prod_{i=1}^{n-1} i^{n-i}$  (see \cite{Hilbert}).  Hence if we define  $f(b_n) = \frac{1}{n+1}$ for any $n\geq 0$ , then the kernel of $f$ does not contain any non-zero ideal and the bilinear form $\beta=f\circ \mu$ is non-degenerate. An algebra with a multiplicative basis as above is for example the polynomial algebra $A=F[x]$. Thus $S=F\times F[x] \times F[x]$ does not satisfy $(\diamond)$.
\end{example}

\section{Acknowledgement} 
The authors would like to thank Mio Iovanov for pointing out  example \ref{example_Mio} and for interesting discussions about an earlier draft of this paper.
This paper started with a visit of the last named author to the University of Porto in September 2011.
The authors also would like to thank Alveri Sant'ana for having discussed the results of this paper during his visit to Porto.
The two first named authors 
are members of CMUP (UID/MAT/00144/2013), which is funded by FCT with national (MEC) and European structural funds (FEDER), under the partnership agreement PT2020.

\end{document}